\documentclass[11pt]{amsart}
\usepackage{amsmath, amsthm,amssymb, amsfonts}

\newcommand{\ben}{{\mathbb N}}
\newcommand{\ber}{{\mathbb R}}
\newcommand{\beq}{{\mathbb Q}}

\newtheorem{theorem}{Theorem}[section]
\newtheorem{lemma}[theorem]{Lemma}

\theoremstyle{definition}
\newtheorem{definition}[theorem]{Definition}
\theoremstyle{remark}

\numberwithin{equation}{section}

\theoremstyle{definition}

\begin{document}
\title{Preservation of notion of C sets near zero over reals}
\author{Kilangbenla Imsong}
\address{Department of Mathematics, Nagaland University, Lumami-798627, Nagaland, India.}
\email{kilangbenla@nagalanduniversity.ac.in}
\author{Ram Krishna Paul}
\address{Department of Mathematics, Nagaland University, Lumami-798627, Nagaland, India.}
\email{rmkpaul@gmail.com, rkpaul@nagalanduniversity.ac.in}
\begin{abstract}
    There are several notions of largeness in a semigroup. N. Hindman and D. Strauss established that if $u,v \in \, \ben$, $A$ is a $u \times v$  matrix with entries from $\beq$ and $\psi$ is a notion of a large set in $\ben$, then $\{\vec{x} \in \ben^v: A\vec{x} \in \psi^u \}$ is large in $\ben^v$.  Among the several notions of largeness, C sets occupies an important place of study because they exhibit strong combinatorial properties. The analogous notion of C set appears for a dense subsemigroup $S$ of $((0, \infty),+)$ called a C-set near zero. These sets also have very rich combinatorial structure. In this article, we investigate the above result for C sets near zero in $\ber^+$  when the matrix has real entries. We also develop a new characterisation of C-sets near zero in $\ber^+$.
\end{abstract}
\maketitle

\section{Introduction}
The famous van der Waerden's theorem \cite{W} of Ramsey Theory  states that if $r,l \in \mathbb{N}$ and $\mathbb{N} = \bigcup_{i=1}^{r} D_i$, then there exists $i \in \{1,2,...,r\}$ and $a,d  \in \mathbb{N}$ with  $\{a,a+d,...,a+ld\} \subseteq D_i$. This powerful combinatorial result can also be expressed using matrices, by saying that the entries of the matrix  $M \Vec{x}$, where 
  \begin{center} 
 $   M = \begin{pmatrix}
      1 & 0\\
      1 & 1\\
      1 & 2\\
      \vdots\\
      1 & l\\
      \end{pmatrix}$
 and 
      $\Vec{x} =\begin{pmatrix}
  a \\
  d
  \end{pmatrix} \in  \mathbb{N}^2$, \\
  \end{center}
  are monochromatic.  A similar expression of Schur's theorem, which states that if $r \in \mathbb{N}$ and $\mathbb{N} = \bigcup_{i=1}^{r} D_i$ then there exists $i \in \{1,2,...,r\}$ and $x,y  \in \mathbb{N}$ with $\{x,y,x+y\} \subseteq D_i$, is that, the entries of the matrix  $M \Vec{x}$, where 
   \begin{center} 
 $   M = \begin{pmatrix}
      1 & 0\\
      0 & 1\\
      1 & 1
      \end{pmatrix}$
and
      $\Vec{x} =\begin{pmatrix}
  x \\
  y
  \end{pmatrix} \in  \mathbb{N}^2$, \\
  \end{center}
  are monochromatic.
  These expressions of the combinatorial results using matrices follows from the following definition.

  \begin{definition} Let $u,v \in \mathbb{N}$ and let $M$ be a $u \times v$ matrix with entries from $\mathbb{Q}$. The matrix $M$ is image partition regular over $\mathbb{N}$ (abbreviated IPR/$\ben$) if and only if whenever $r \in \mathbb{N}$ and $\mathbb{N} = \bigcup_{i=1}^{r} C_i$, there exists $i\in \{1,2,...,r\}$ and $\Vec{x} \in \mathbb{N}^v$ such that $M\Vec{x} \in C_i^u$ . 
  \end{definition}
Thus the above two theorems is the assertion that the matrix $M$ is image partition regular over $\ben$.

The first characterisations of matrices with entries from $\mathbb{Q}$ that are IPR/$\mathbb{N}$ were obtained in \cite{HLb} and over the years other characterisations have been obtained, which can be found in detail in \cite{HSc}. Among these, the ones involving the notion of \textit{central} sets have been a focal point of study primarily because of the rich combinatorial structure they possess.   Central subsets of $\mathbb{N}$ were first introduced by Furstenberg defined in terms of notions of topological dynamics and he has given and proved the Central Sets Theorem which we state below. We let $\mathcal{P}_f(X)$ to denote any finite collection of subsets of a set $X$.
\begin{theorem} Let $l \in \mathbb{N}$ and for each $i \in \{1,2, \cdots, l\}$, let $\{y_{i,n}\}_{n=1}^{\infty}$ be a sequence in $\mathbb{Z}$. Let $C$ be central subset of $\mathbb{N}$. Then there exists sequences $\{a_n\}_{n=1}^{\infty}$ in $\mathbb{N}$ and $\{
H_n\}_{n=1}^{\infty}$ in $\mathcal{P}_f(\mathbb{N})$ such that 

(1) for all $n$, max $H_n<$ min $H_{n+1}$ and

(2) for all $F \in \mathcal{P}_f(\mathbb{N})$, and all $i \in \{1,2,...,l\}$,
\begin{center}
    $\sum_{n \in F} ( a_n + \sum_{t \in H_n}y_{i,t})\in C$.
\end{center}
\end{theorem}
\begin{proof}
 \cite[Proposition 8.21]{F}
\end{proof}
Central sets were also defined using the algebra of Stone $\check{C}$ech compactification of a discrete semigroup which can be found in  \cite[Definition 4.42]{HSc}. The equivalence of the notions of \textit{central} and \textit{dynamically central} was established by V. Bergelson and N. Hindman in \cite{BH}. 
The following theorem is the characterisation of image partition regular matrices involving central sets.

\begin{theorem} \label{th1}Let $u,v \in \mathbb{N}$ and let $A$ be a $u \times v$ matrix with entries from $\mathbb{Q}$. The following statements are equivalent.

(a) $A$ is image partition regular over $\ben$.

(b) For every central set $C$ in $\mathbb{N}$, there exists  $\Vec{x} \in \mathbb{N}^v$ such that $A\Vec{x} \in C^u$. 

(c)  For every central set $C$ in $\mathbb{N}$, \{$\Vec{x} \in \mathbb{N}^v$ :  $A\Vec{x} \in C^u$ \} is central in $ \mathbb{N}^v$.
    
\end{theorem}
\begin{proof} \cite[Theorem 15.24]{HSc}
\end{proof}

The sets which satisfy the conclusion of Central Sets Theorem are called C-sets and hence the natural question arises as to whether Theorem \ref{th1}  holds true if \textit{central set} is  replaced by \textit{C set}. This has been answered in affirmative and Hindman and Strauss has proved it in \cite [Theorem 1.4]{HSb}.\\
For any dense subsemigroup  $S$ of $(\mathbb{R},+)$, it has been observed that there are subsets of $\mathbb{R}$ living near zero which also satisfy a version of the Central Sets Theorem. These sets have been defined as \textit{central sets near zero} in \cite{HL} and the corresponding notion of \textit{Image Partition Regularity Near Zero} has been developed in \cite{DH}.
\begin{definition}\label{thtB} Let $S$ be a subsemigroup
of $({\mathbb R},+)$ with $0\in$ $c\ell S$, let $u,v\in\ben$, and let
$A$ be a $u\times v$ matrix with entries from $\beq$.
Then $A$ is {\em image partition regular over $S$ near zero\/}
(abbreviated IPR/$S_0$)
if and only if, whenever $S\setminus\{0\}$ is finitely
colored and $\delta>0$, there exists $\vec x\in S^v$
such that the entries of $A\vec x$ are monochromatic and lie
in the interval $(-\delta,\delta)$.
\end{definition}
In the same paper \cite{DH}, image partition regularity near zero over various dense subsemigroups have been considered. Of notable significance for our work is the result that for finite matrices, image partition regular matrices over $\mathbb{R}$ and image partition regular matrices near zero over $\mathbb{R}^+$ are same which has been shown by one of the authors  in \cite [Theorem 2.3]{DPb}. 
The notion of C sets have also been extended to C sets near zero in \cite{BT}. In this paper we will give  a new characterisation of C-sets near zero in $\ber^+$ and prove the following theorem. 

\begin{theorem}\label{th2} Let $u,v \in \mathbb{N}$ and let $M$ be a $u \times v$ matrix with entries from $\mathbb{R}$. The following statements are equivalent.\\
\hspace{1 cm}(a) $M$ is image partition regular over $\mathbb{R}^+$.\\
\hspace{1 cm} (b) For every C-set near zero $C$ in $\mathbb{R}^+$, there exists  $\Vec{x} \in (\mathbb{R}^+)^v$ such that $M\Vec{x} \in C^u$. \\
\hspace{1 cm} (c)  For every C-set near zero $C$ in $\mathbb{R}^+$, \{$\Vec{x} \in (\mathbb{R}^+)^v$ :  $M\Vec{x} \in C^u$ \} is a C-set near zero in $ (\mathbb{R}^+)^v$.
\end{theorem}

\section{Preliminary Results}
 
For this paper, we need the algebraic structure of the Stone-Cech Compactification, denoted by $\beta S$, of a discrete semigroup $(S,+)$. We let $\beta S=\{ p: p$ is an ultrafilter on $S\}$ and identify the principal ultrafilters on $S$ with the points of $S$ and so we may assume $S \subseteq \beta S$. Given $A \subseteq S$, $\overline{A}=\{p \in \beta S: A \in p\}$, where $\overline{A}$  denotes the closure of $A$ in $\beta S$ and $\{\overline{A}: A \subseteq S\}$ is a basis for the topology of $\beta S$. The operation $+$ on $S$ extends to an operation, also denoted by $+$, on $\beta S$ so that $\beta S$ is a right topological semigroup with $S$ contained in the topological centre of $\beta S$. That is, for each $p \in \beta S$, the function $\rho_p: \beta S \rightarrow \beta S$ defined by $\rho_p(q)=q+p$ is continuous and for each $x \in S$, the function $\lambda_x: \beta S \rightarrow\beta S$ defined by $\lambda_x(q)=x+q$ is continuous. Given $p,q \in \beta S, A \in p+q$ if and only if $\{x \in S: -x+A \in q\} \in p$, where $-x+A=\{y \in S: x+y \in A\}$. Since $\beta S$ is a compact Hausdorff right topological semigroup, it contains idempotents and a smallest two sided ideal, denoted by $K(\beta S)$, which is the union of all of the minimal left ideals of $\beta S$ and also the union of all the minimal right ideals of $\beta S$. The detailed algebraic structure of $\beta S$ is available in \cite{HSc}.

In this section, we prove some results for subsets $E$ of $S^v, v>1$ where $S$ is a dense subsemigroup of $((0, \infty),+)$. We consider $S$ to be a dense subsemigroup of $((0,\infty),+)$, where $\textit
{dense}$ means with respect to usual topology on $((0,\infty),+)$ and when we use the Stone Cech Compactification of such a semigroup $S^v, v \ge 1$ then we deal with $S_d^v$ which is the set $S^v$ with the discrete topology. \\
In this paper we shall prove results for $v=2$ which can be easily adapted to prove the results are true for any $v >2$.\\
We start with the following definition.
\begin{definition} \label{defA} Let $S$ be a dense subsemigroup of $((0,\infty),+)$ and $v>1$. Then 
$0^+(S^v)=\{ p\in\beta (S_d^v): (\forall \epsilon>0)(((0,\epsilon) \cap S)^v \in p)\}$.  
\end{definition}
\begin{lemma} \label{l2} Let $S$ be a dense subsemigroup of $((0,\infty),+)$, then $0^+(S^v)$, $v >1$ is a compact right topological subsemigroup. 
\end{lemma}
\begin{proof}
    Proof follows in a similar fashion as in \cite[lemma 2.5]{HL}. 
\end{proof}
As a consequence of Lemma \ref{l2}, $0^+(S^2)$  contains idempotents. It was noted that $0^+(S) \cap K(\beta S_d) = \emptyset$ in \cite {HL}.We now prove the following.  

\begin{lemma} \label{l3} Let $S$ be a dense subsemigroup of $((0,\infty),+)$, then \\
$0^+(S^2) \bigcap K(\beta S_d^2)=\emptyset$.
\end{lemma}
\begin{proof}
Consider the identity function $i:S\times S \rightarrow S \times S \subseteq [0,\infty] \times [0, \infty]$.\\
 Here, $[0,\infty]$ is a topological space with respect to the order topology and  $[0,\infty]\times [0,\infty]$ is compact with respect to the product topology. Also, $[0,\infty] \times [0, \infty]$ is a semigroup with respect to the pointwise addition. Further, $\rho_{\vec{x}}: [0,\infty] \times [0,\infty] \rightarrow [0, \infty] \times [0,\infty]$ defined by $\rho_{\vec{x}}(\vec{y})=\vec{y}+\vec{x}$ for all $\vec{y} \in [0,\infty] \times [0,\infty]$ is continuous for all $\vec{x} \in [0,\infty] \times [0,\infty]$. Therefore, $ [0,\infty]\times[0,\infty]$ is a compact right topological semigroup. Since $ i(S\times S)$ is contained in the topological centre of $[0,\infty] \times [0,\infty]$, therefore,  the identity function can be continuously extended to a homomorphism from $\beta(S^2_d)$ to $[0,\infty] \times [0, \infty]$.\\
Let $\alpha : \beta S^2_d \rightarrow [0,\infty] \times [0, \infty]$ be the continuous extension.\\ 
Let $B(S^2) = \{p\in \beta S^2_d : \alpha(p) \notin \{(x,\infty),(\infty, x): x \in [0,\infty]\}\}$ and let $\vec{x} =(x_1,x_2) \in [0,\infty] \times [0,\infty]$. We define\\
$\vec{x}^+ = (x_1,x_2)^+ = \{p \in B(S^2): \alpha(p)=(x_1,x_2)$ and $((x_1,\infty) \cap S) \times ((x_2, \infty) \cap S) \in p\}$\\
Then $p \in \vec{x}^+$ if and only if for every $\epsilon >0,((x_1,x_1+\epsilon) \cap S \times (x_2,x_2+\epsilon) \cap S) \in p$. In particular, when $x_1=x_2=0$, then $\vec{x}^+=0^+(S^2)$.\\
Let $A = \{p \in \beta S^2_d : \alpha(p) \in \{ (x,\infty), (\infty, x): x \in  [0, \infty]\}\}$.\\
Then we claim that $A \neq \emptyset$. Suppose $(x, \infty) \notin \alpha(\beta S_d^2)$ for all $x \in [0,\infty]$. Since $\alpha(\beta S_d^2)$ is closed, therefore $(x,\infty)$ is not a limit point of $\alpha(\beta S_d^2)$. Therefore, there exists open set $(a,b) \times (c, \infty]$, for some $a,b,c \in [0,\infty]$ containing $(x,\infty)$ such that \begin{align*} (a,b) \times (c, \infty] \cap \alpha(\beta S_d^2) = \emptyset \\
 \implies (a,b) \times (c, \infty] \cap \alpha( \overline{S^2})= \emptyset\\
\implies (a,b) \times (c, \infty] \cap \overline{\alpha( S^2)} = \emptyset\\
\implies (a,b) \times (c, \infty] \cap \alpha( S^2) = \emptyset\\
\implies (a,b) \times (c, \infty] \cap  S^2 = \emptyset
\end{align*}
Since $S^2$ is a dense subsemigroup of $((0,\infty) \times (0,\infty), +)$, therefore any $(e,f)$ where $a<e<b, f>c$ is a limit point of $S^2$. So we get a contradiction. And hence the claim. Clearly,  $0^+(S^2) \cap A = \emptyset$. Also $A$ is a left ideal of $\beta S^2_d$ and every minimal left ideal of $\beta S^2_d$ is in $A$ and hence $K(\beta S^2_d)\subseteq A$. Thus  $K(\beta S^2_d) \cap 0^+(S^2) = \emptyset$.
\end{proof}

The proof of Lemma \ref{l3}  can be adapted to prove the following.

\begin{lemma}
     Let $S$ be a dense subsemigroup of $((0,\infty),+)$ then for $v \ge 1$,\\  $0^+(S^v) \bigcap K(\beta S^v) = \phi$.
\end{lemma}

    Let $S$ be a dense subsemigroup of $((0,\infty),+)$. Consider $S^v, v \ge 1$. The set of sequences in $S^v$ converging to zero is denoted by $\mathcal{T}_0^v$.

The notion of $J$-set near zero for any $E \subseteq S$, where $S$ is a dense subsemigroup of $((0,\infty),+)$ was defined in \cite{BT} . We extend this notion and define $J$-set near zero for any $E \subseteq S^v, v>1$.
\begin{definition} \label{defC}Let $S$ be a dense semigroup of $((0,\infty),+)$ and $E\subseteq S^v$, where $v\ge 1$. Then $E$ is a J-set near in $S^v$ zero iff whenever $F \in \mathcal{P}_f(\mathcal{T}_0^v)$ and $\delta > 0$, there exists $\vec{a} \in S^v \cap (0,\delta)^v$ and $H \in \mathcal{P}_f(\mathbb{N})$ such that for each $\vec{f} \in F,\, \vec{a} +  \sum_{t \in H}\vec{f}(t) \in E$.    
\end{definition}

\begin{definition} \label{defB} Let $S$ be a dense subsemigroup of $((0,\infty),+)$. Consider $S^v, v \ge 1$. Then $J_0(S^v)=\{ p\in 0^+(S^v): \forall E \in p, E $ is a J-set near zero \}.     
\end{definition}

By \cite[Theorem 3.10]{BT}, $J_0(S)$ is a compact two sided ideal of $0^+(S)$ and in a similar fashion we can prove that $J_0(S^v)$ is a two sided ideal of $0^+(S^v)$ for any $v > 1$.   Lemma \ref{l4} below proves that $J_0(S^2)$ is a closed subset of $0^+(S^2)$ from which it follows that $J_0(S^2)$ is a compact two sided ideal of $0^+(S^2)$. 
\begin{lemma} \label{l4} ${J_0}(S^2)$ is a closed subset of $0^+(S^2)$. 
\end{lemma}
\begin{proof}
    Let $q \in 0^+(S^2)$ such that $q \notin J_0(S^2)$. Then $\exists  \; B \in q$ such that $B$ is not a $J$-set near zero in $S^2$. Let $B=E_1 \times E_2$ for some $E_1$ and $E_2$ which are subsets of $S$. We will show that $q$ is not a limit point of $J_0(S^2)$.\\
Now, $\{\overline{U} \cap 0^+(S^2): U \subseteq S^2\}$ is a basis of $0^+(S^2)$
\begin{align*}
    \overline{E_1 \times E_2} \bigcap 0^+(S^2) &= \{r \in \beta S^2 : r \in \overline{E_1 \times E_2} \text{ and } r \in 0^+(S^2)\}\\
    &= \{r \in \beta S^2 : \forall \, \epsilon >0,  ((0,\epsilon) \times (0,\epsilon)) \bigcap(E_1 \times E_2) \in r\}.
\end{align*} 
Now, $E_1 \times E_2$ is not a $J$-set near zero. Therefore, for all $\epsilon >0, ((0, \epsilon) \times (0,\epsilon))\bigcap (E_1 \times E_2)$ is also  not a $J$-set near zero.
Now, $\overline{E_1 \times E_2} \bigcap 0^+(S^2)$ is a neighborhood of $q$ such that
$\overline{E_1 \times E_2} \bigcap 0^+(S^2) \bigcap J_0(S^2)= \emptyset$.
Therefore, $q$ is not a limit point of $J_0(S^2)$.

\end{proof}
The subsets of $S^2$ which are J-sets near zero are partition regular, the proof of which easily follows from  \cite[Lemma 3.9]{BT}. 
\begin{lemma} \label{l5}
Let $S$ be a dense subsemigroup of $((0,\infty),+)$ and let $E_1$ and $E_2$ be subsets of $S^2$. If $E= E_1 \cup E_2$ and $E$ is a J-set near zero then either $E_1$ is a J-set near zero in $S^2$ or $E_2$ is a J-set near zero in $S^2$.
\end{lemma}
\begin{proof}
     The proof of \cite[Lemma 3.9]{BT} can be adapted.
\end{proof}
The algebraic characterisation of $J$-sets near zero and $C$- sets near zero for subsets of $S^2$ are given below.
\begin{theorem} \label{thA} Let $S$ be a dense subsemigroup of $((0,\infty),+))$ and $E \subseteq S^2$. Then $\overline{E} \cap J_0(S^2) \neq \phi$ iff $E$ is a J-set near zero.
\end{theorem}
\begin{proof}
    Necessity .
    If $\overline{E} \cap J_0(S^2) \neq \phi, \exists \; p \in 0^+ (S^2)$ such that $E \in p$ and $p \in J_0(S^2)$, it implies $E$ is a J-set near zero. 

    Sufficiency.
    Let $E$ be a J-set near zero.
    To prove that $\overline{E} \cap J_0(S^2) \neq \emptyset$, we need to show that $\exists \; p \in 0^+ (S^2)$ such that $E \in p$ and $\forall \; B \in p, B$ is a J-set near zero in $S^2$.\\
    By Lemma \ref{l5}, J-sets near zero are partition regular. So if $E$ is a J-set near zero, by \cite [Theorem 3.11]{HSc}, there is some $p \in 0^+ (S^2)$ such that $E \in p$ and for every $B \in p, B$ is a J-set near zero.
\end{proof}

\begin{theorem} \label{thB} Let $S$ be a dense subsemigroup of $((0,\infty),+))$ and $E \subseteq S^2$. Then $E$ is a $C$-set near zero if and only if there is an idempotent in  $\overline{E} \cap J_0(S^2)$.
\end{theorem}
\begin{proof}
    This follows directly from \cite[Theorem 3.15]{BT}
\end{proof}

\section{Preservation of the notion of C-sets near zero}
In this section we will prove that the pre-image of a $C$ set near zero is again a $C$ set near zero via image partititon regular matrices. We first prove some results and lemmas which will help us do so. We recall the following definition:
\begin{definition}
$0^+(\ber^+)=\{ p\in\beta (\ber_d^+): (\forall \epsilon>0)((0,\epsilon) \in p)\}$.  
\end{definition}
We let $\mathcal{R}_0$ to denote the set of sequences in $\ber$ converging to zero and  $\mathcal{R}_0^+$ to denote the set of sequences in $\ber^+$ converging to zero.\\

The following lemma allows the characterisation of J-sets near zero to be preserved even when we allow $f \in F$ to take on negative values in $\ber$.
\begin{lemma} \label{l6} 
   Let $C$ be a $J$-set near zero in $\mathbb{R}^+$. For each  $F \in \mathcal{P}_f(\mathcal{R}_0)$ and $\delta > 0, $ there exists $ a \in (0,\delta)$ and H $\in \mathcal{P}_f(\mathbb{N})$ such that for each $f \in F,a +  \sum_{t \in H}f(t) \in C$.
\end{lemma}
\begin{proof}
   Let $F \in \mathcal{P}_f(\mathcal{R}_0)$ and $\delta > 0$. For each $f \in F$ and $n \in \ben $, we choose $b(n) \in \ber ^+ \cup \{0\}$ such that $f(n) + b(n) \in \ber^+$ and $\{b(n)\}_{n=1}^\infty$ converges to zero. Then there is a subsequence $\{b(n_k)\}_{k=1}^\infty$ of $\{b(n)\}_{n=1}^\infty$ such that $ \sum_{k=1}^\infty b(n_k)$ converges and $ \sum_{k=1}^\infty b(n_k) < \frac{\delta}{2}$.  Define $h_f(n_k)=b(n_k)+f(n_k)$. Then $ h_f \in \mathcal{R}_0^+ $. Since $C$ is a J-set near zero\\
  Therefore, $\exists \: p \in 0^+(\ber^+)$ such that $p \in J_0(\ber^+) \cap \overline{C}$.
  
  $\implies \exists \, p \in 0^+(\ber^+)$ such that for any $\epsilon>0, (0,\epsilon) \cap C \in p$
  
Let $C' = C \cap (0,\delta)$, then $C'$ is a $J$-set near zero.
  Let $F'=\{\{h_f(n_k)\}_{k=1}^\infty: f \in F\}$. For $ F' \in \mathcal{P}_f(\mathcal{R}_0^+)$ and $\frac{\delta}{2} > 0,\, \exists \; a \in (0, \frac{\delta}{2})$ and $H\in \mathcal{P}_f(\ben)$ such that for all $f \in F$, $a + \sum_{k \in H}h_f(n_k) \in C'$ 
  
  i.e., $a +  \sum_{k \in H}b(n_k) + \sum_{k \in H}f(n_k) \in C \cap (0,\delta)$\\
  Let $a'= a +  \sum_{k \in H}b(n_k)$, then $a' \in (0,\delta)$. Thus, there exists $a' \in (0,\delta)$ and $H\in \mathcal{P}_f(\ben)$ such that for all $f \in F, a' +  \sum_{t \in H}f(t) \in C$. 
\end{proof}

\begin{lemma}\label{l7} Let $u,v \in \ben$ and let $M$ be a $u \times v$ matrix with entries from $\ber$. Given $C \subseteq \ber ^+$, let $D(C)= \{ x \in (\ber ^+)^v: M \Vec{x} \in C^u\}$. Let $p$ be an idempotent in $J_0(\ber^+)$. If for every $C \in p, D(C)$ is a $J$-set near zero in $(\ber^+)^v$, then for every $C \in p, D(C)$ is a $C$-set near zero in $(\ber^+)^v$. 
\end{lemma}
\begin{proof} Define $\phi : (\ber ^+)^v \rightarrow \ber^u$ by $\phi(\Vec{x}) = M\Vec{x},  \forall \; \vec{x} \in (\ber^+)^v$ and let $\Tilde{\phi}: \beta (\ber^+)^v \rightarrow (\beta \ber)^u$ be its continuous extension which is also a homomorphism. Let $C \subseteq \ber^+$ . We rewrite, $D(C)= \{ x \in (\ber ^+)^v: M \Vec{x} \in C^u\}= \phi ^{-1}[C^u]$. Now, to show that for every $C \in p, \phi ^{-1}[C^u]$ is a $C$- set near zero in $(\ber^+)^v$, we need to show that for every $C \in p $, there is an idempotent in $\overline{\phi ^{-1}[C^u]} \bigcap J_0((\ber^+)^v)$.  Given that, for every  $C \in p, \phi ^{-1}[C^u]$ is a $J$-set near zero in $(\ber^+)^v$. Therefore, for every $C \in p, \overline{\phi ^{-1}[C^u]} \cap J_0((\ber^+)^v) \neq \emptyset $. Now for every $C \in p, \overline{\phi ^{-1}[C^u]} \cap J_0((\ber^+)^v)$ is compact in $\beta ((\ber^+)^v)$ and hence closed in $0^+((\ber^+)^v)$. Since,  $0^+((\ber^+)^v)$ is compact and $\{\overline{\phi ^{-1}[C^u]} \cap J_0((\ber^+)^v): C \in p\}$ is a collection of closed sets having finite intersection property, therefore, $\bigcap_{C \in p} (\overline{\phi ^{-1}[C^u]} \cap J_0((\ber^+)^v)) \neq \emptyset$. We claim that $\bigcap_{C \in p} \overline{\phi ^{-1}[C^u]} \subseteq \widetilde{\phi}^{-1}({\{\overline{p}}\})$. Let $r \in \beta ((\ber^+)^v)$ such that $r \in \bigcap_{C \in p} \overline{\phi ^{-1}[C^u]} \; \forall \; C \in p$. Then $ \phi ^{-1}[C^u] \in r \; \forall \; C \in p $. Let $\phi ^{-1}[C^u]=B$ for some $B \in r \implies \phi(\phi^{-1}[C^u])=\phi(B) \subseteq \overline{C}^u \; \forall \; C \in p$ and some $B \in r$. Suppose $r \notin \tilde{\phi}^{-1}(\{\vec{p}\}) \implies$ there exists a neighbourhood , say, $W$ of $\Tilde{\phi}(r)$ and a neighbourhood $\overline{D}^u$ of $\vec{p}$ such that $W \cap \overline{D}^u = \emptyset $ where $D \in p$. For the neighbourhood $W$ of $\Tilde{\phi}(r)$, there exists a neighbourhood $\overline{E}$ of $r$, where $E \in r$, such that $\tilde{\phi}(\overline{E}) \subseteq W$. Now, $\tilde{\phi}(E) \subseteq \tilde{\phi}(\overline{E})\subseteq W. \implies \phi(E) \subseteq W \implies \phi(E \cap B) \subseteq\phi(E) \subseteq W$. In particular, $\phi(E \cap B) \subseteq\phi(B) \subseteq \overline{D}^u$. Since $E \in r, B \in r \implies E \cap B \in r \implies \phi(E \cap B)\neq \emptyset \implies W \cap \overline{D}^u\neq \emptyset. $ This is a contradiction and hence our claim. 
So, $ \widetilde{\phi}^{-1}({\{\overline{p}}\}) 
\cap J_0((\ber^+)^v)) \neq \emptyset$. Thus,  $ \widetilde{\phi}^{-1}({\{\overline{p}}\}) \cap J_0((\ber^+)^v))$ is a compact subsemigroup of $0^+((\ber^+)^v)$. So, pick an idempotent $q \in\widetilde{\phi}^{-1}({\{\overline{p}}\}) \cap J_0((\ber^+)^v))$. Given, $C \in p, \overline{C}^u$ is a neighbourhood of $\overline{p}$ where $\overline{p}=(p,\cdots, p)$. Pick $B \in q$ such that $\widetilde{\phi}[B] \subseteq \overline{C}^u$. therefore, $\phi[B] \subseteq C^u$. This shows that $B \subseteq \phi^{-1}[C^u]$.
\end{proof}
\begin{definition}
Let $M$ be a $u \times v$ matrix with entries from $\mathbb{R}$. Then $M$ is a \textit{first entries matrix} if:\\
    (1) No row of $M$ is $\Vec{0}$.\\
    (2) The first nonzero entry of each row is positive.\\
    (3) If the first nonzero entry of any two rows occur in the same column, then those entries are equal.
\end{definition}
If $M$ is a first entries matrix and $c$ is the first nonzero entry of some row of $M$, then $c$ is called a \textit{first entry} of $M$.

\begin{theorem}\label{th3a}
    Let $u,v \in \mathbb{N}$,  and let $M$ be a $u \times v$ matrix with entries from $\mathbb{R}$. The following statements are equivalent:\\
    (a) $M$ is image partition regular over $\mathbb{R}^+$.\\
    (b) There exists $m \in \ben$ , and  a $u \times m$ first entries matrix $B$ such that for all $\vec{y} \in (\mathbb{R}^+)^m$ there exists $\vec{x} \in (\mathbb{R}^+)^v$ such that $M\vec{x}=B\vec{y}$. 
\end{theorem}
\begin{proof}
    (a) $\implies$ (b). Follows from \cite [Theorem 4.1]{H}.\\
    (b) $\implies$ (a). Follows from \cite [Theorem 4.1]{H}.
\end{proof}

\begin{theorem}\label{th3}
    Let $u,v \in \mathbb{N}, M$ is a $u \times v$ matrix with entries from $\mathbb{R}$. The following are equivalent:\\
    (a) $M$ is IPR/$\mathbb{R}^+$.\\
    (b) Given any column $\vec{c} \in \ber^u$, the matrix  $M'=\begin{pmatrix}
        M & \vec{c}
    \end{pmatrix}$ is IPR/$\ber^+$.
\end{theorem}

\begin{proof}
    (a) $\implies$ (b) Given, $\vec{c} \in \mathbb{R}^u$ and $M'=\begin{pmatrix}
        M & \vec{c}
    \end{pmatrix}$. Since $M$ is \textit{IPR/$\mathbb{R}^+$}, by the above Theroem \ref{th3a}, there exists $m \in \ben$ and a $u \times m$ first entries matrix $B$ such that for all $\vec{y} \in (\mathbb{R}^+)^m$ there exists $\vec{x} \in (\mathbb{R}^+)^v$ such that $M\vec{x}=B\vec{y}$. Take $\vec{w} \in (\ber^+)^{v+1}$ such that the first $v$ entries of $\vec{w}$ consist of the first $v$ entries of $\vec{x}$ and the $(v+1)^{th}$ entry is any positive real number, say $z$. Then, $M'\vec{w}=\begin{pmatrix}
        M &\vec{c}
    \end{pmatrix}\begin{pmatrix}
        \vec{x}\\
        z
    \end{pmatrix}= M\vec{x}+\vec{c} z=B\vec{y} + \vec{c}z=\begin{pmatrix}
        B & \vec{c}
    \end{pmatrix}\begin{pmatrix}
        \vec{y}\\
        z
    \end{pmatrix}$ for any $\vec{y} \in (\ber^+)^m$. Hence, $M'\vec{w} =\begin{pmatrix}
        B & \vec{c}
    \end{pmatrix}\vec{z}$ for any $\vec{z} \in (\ber^+)^{m+1}$ and hence $M'\vec{w}=B'\vec{z}$, where $B'$ is a $u \times (m+1)$ first entries matrix. Thus, by the above Theroem \ref{th3a} again,  $M'$ is also \textit{IPR}/$\mathbb{R}^+$.\\
    (b) $\implies$ (a) is trivial.
\end{proof}

\begin{definition}
     Let $(S,+)$ be a semigroup and $A \subseteq S, p \in \beta S$. Then $A^\star(p)=\{s \in A: -s +A \in p\}$
\end{definition}
We shall write $ A^\star$ instead of $A^\star(p)$.
\begin{lemma} \label{l8}
    Let $(S, +)$ be a semigroup, let $p+p=p \in \beta S$ and let $A \in p$. For each $s \in A^\star, -s+A^\star \in p $.
\end{lemma}
\begin{proof}
    \cite[Lemma 4.14]{HSc}
\end{proof}

\begin{lemma} \label{l9} Let $u,v,d \in \mathbb{N}$, let $M$ be a $u \times v$ matrix with entries from $\mathbb{R}$. Let $B$ be a $u \times d$ matrix with entries from $\mathbb{R}$ and assume that for any $\delta>0$, whenever $C$ is a C-set near zero in $\mathbb{R}^+$, and $\{\Vec{b}(n)\}_{n=1}^{\infty}$ is a sequence in  $(\mathbb{R}^+)^d$ which converges to zero, there exists $ \Vec{x} \in  (0, \delta)^v$ and $\Vec{y} \in FS(\{\Vec{b}(n)\}_{n=1}^{\infty})$ such that all entries of  $M\Vec{x} + B\Vec{y}$ are in $C$. Then for any $\delta > 0$, whenever $C$ is a C-set near zero in $\mathbb{R}^+$ and $\{\Vec{b}(n)\}_{n=1}^{\infty}$ is a sequence in  $(\mathbb{R}^+)^d$ which  converges to zero, there exists a sequence $\{\Vec{x}(n)\}_{n=1}^{\infty}$ in $(0, \delta)^v$
such that $\{\Vec{x}(n)\}_{n=1}^{\infty}$ converges to zero  and a sequnce $\{H_n\}_{n=1}^\infty$ in  $\mathcal{P}_f(\ben)$ with max $H_n < min H_{n+1}$ for each $n \in \ben$ such that $FS(\{M\Vec{x}(n) + B \sum_{t \in H_n}\Vec{b}(t)\}_{n=1}^{\infty}) \subseteq C^u$.  
\end{lemma}
\begin{proof}
Let $C$ be a C-set near zero in $\mathbb{R}^+$ 
Therefore, there exists an idempotent $p \in J_0(\ber^+) \cap \overline{C}$. Let $C^\star=\{x \in C: -x +C \in p\}$. Then $C^\star \in p$. Also by Lemma \ref{l8}, if $x \in C^\star$ then $-x+C^\star \in p$. Now $C^\star$ is also a C-set near zero in $\ber^+$. So, pick $\Vec{x}(1) \in (0,\delta)^v$ and $\Vec{y}(1) \in FS(\{\Vec{b}(n)\}_{n=1}^{\infty})$ such that all entries of  $M\Vec{x}(1) + B\Vec{y}(1)$ are in $C^\star$. Pick $H_1 \in \mathcal{P}_f(\ben)$ such that $\Vec{y}(1) =\sum_{t \in H_1}\Vec{b}(t)$, i.e., the result is true for $n=1$ for $\{\vec{x}(t)\}_{t=1}^{n}$.\\
Let $n \in \ben$ and let $\Vec{x}(t)=(x_1(t), x_2(t),...,x_v(t))$ and assume that we have chosen $\{\Vec{x}(t)\}_{t=1}^{n}$ in $(0,\delta)^v$ and $\{H_t\}_{t=1}^{n}$ in  $\mathcal{P}_f(\ben)$  such that for each $i \in \{1,2,...,v\}, x_i(t)<\frac{\delta}{t}$ for each  $t \in \{1,2,...,n\}$ and max $H_t <$ min $H_{t+1}$ for all $t \in \{1,2,..., n-1\}$ and for each $F \in \mathcal{P}_f(\{1,2,\cdots, n\})$, all entries of $\sum _{t \in F}(M\Vec{x}(t) + B \sum_{k \in H_t}\Vec{b}(k))$ are in $C^\star$. Let $\vec{z}(t)=M\Vec{x}(t) + B \sum_{k \in H_t}\Vec{b}(k)$. We wish to find $\vec{x}(n+1) \in (0,\delta)^v$ and $H_{n+1} \in \mathcal{P}_f(\ben)$ with  max $H_n <$ min $H_{n+1}$  and for each $F \in \mathcal{P}_f(\{1,2,\cdots, n+1\}), \sum _{t \in F}(M\Vec{x}(t) + B\sum_{k \in H_t}\Vec{b}(k)) \subseteq C^u$. Let $E= FS(\{\vec{z}(t)\}_{t=1}^n)$. We want  some  $\vec{x}(n+1) \in (0,\delta)^v$ and $H_{n+1} \in \mathcal{P}_f(\ben)$ with  max $H_n <$ min $H_{n+1}$ such that all entries of  $\vec{z}(n+1)=(M\Vec{x}(n+1) + B \sum_{k \in H_{n+1}}\Vec{b}(k))$ are in $ C^\star$ and also such that $\vec{a}+ \vec{z}(n+1) \in (C^\star)^u \; \forall \; \vec{a} \in E$.
i.e., we require $\vec{z}(n+1) \in -\vec{a}+ (C^\star)^u \; \forall \; \vec{a} \in E$.\\
i.e., we want  $\vec{z}(n+1) \in (\bigcap_{\vec{a} \in E} -\vec{a}+ (C^\star)^u) \cap C^\star$ such that $\vec{z}(n+1)=(M\Vec{x}(n+1) + B \sum_{k \in H_{n+1}}\Vec{b}(k))$ for some  $\vec{x}(n+1) \in (0,\delta)^v$ and $H_{n+1} \in \mathcal{P}_f(\ben)$ with  max $H_n <$ min $H_{n+1}$. Take $C_1= (\bigcap_{\vec{a} \in E} -\vec{a}+ C^\star) \cap C^\star$. Then $C_1$ is a C set near zero. Therefore, by the given condition, there exists $\vec{x}(n+1) \in (0, \frac{\delta}{n+1})^v, \vec{y} \in FS(\{\vec{b}(t)\}_{t=m}^\infty$) where $m$= max $H_n +1$ such that $M\vec{x}(n+1) + B \vec{y} \subseteq C_1^u$. Take $H_{n+1} \in \mathcal{P}_f(\mathbb{N})$ such that $\vec{y}= \sum_{t \in H_{n+1}}\vec{b}(t)$ then
$M\vec{x}(n+1)+ B \vec{y} \subseteq ((-\vec{a} + (C^\star)^u)$ for all $\vec{a} \in E$.\\
i.e., $ \vec{a} + M\vec{x}(n+1)+ B \vec{y} \subseteq (C^\star)^u$ for all $\vec{a} \in E$.\\ 
i.e., FS $(\{M\Vec{x}(t) + B \sum_{k \in H_t}\Vec{b}(k)\}_{t=1}^{n+1}) \subseteq C^u$.\\
\end{proof}
For a dense subsemigroup $S$ of $((0,\infty),+)$,  $IP$ sets near zero have been defined  for $A \subseteq S$ in \cite[Definition 2.1]{DP}. We introduce the following definition for any $A \subseteq S^v, v \ge 1$.

\begin{definition} \label{defC}Let $S$ be a dense subsemigroup of $((0,\infty),+)$ and $E\subseteq S^v$, where $v\ge 1$. Then $E$ is a weak $IP$-set near zero in $S^v$ if there exists a sequence $\{\Vec{x}(n)\}_{n=1}^{\infty}$ in  $S^v$ which converges to zero and $FS(\{ \Vec{x}(n)\}_{n=1}^\infty) \subseteq E$.   
\end{definition}

\begin{theorem} \label{th4}
  Let $u,v,d \in \mathbb{N}$, let $M$ be a $u \times v$ matrix with entries from $\mathbb{R}$ which is $IPR/\ber^+$. Let $B$ be a $u \times d$ matrix with entries from $\mathbb{R}$. Let $C$ be a $C$-set near zero in $\mathbb{R}^+$ and $U$ be a weak $IP$-set near zero in $(\ber^+)^d$, then for any  $ \delta >0$, there exists $\Vec{x} \in  (0, \delta)^v$ and $\Vec{y} \in U$ such that all the entries of  $M\Vec{x} + B\Vec{y}$ are in $C$.   
\end{theorem}

\begin{proof}
       Since,  $U$ is a weak $IP$ set near zero in $(\ber^+)^d$, there exists a sequence $\{\Vec{b}(n)\}_{n=1}^{\infty}$ in  $(\ber^+)^d$ which converges to zero and $FS(\{ \Vec{b}(n)\}_{n=1}^\infty) \subseteq U$. Therefore, it is enough to show that given $\delta >0$, a sequence $\{\Vec{b}(n)\}_{n=1}^{\infty}$ in  $(\ber^+)^d$ which converges to zero, there exists $\Vec{x} \in  (0, \delta)^v$ and $\Vec{y} \in FS(\{ \Vec{b}(n)\}_{n=1}^{\infty})$ such that all the entries of  $M\Vec{x} + B\Vec{y}$ are in $C$. 
       
    We shall first prove our theorem for the case in which $M$ is a first entries matrix.
    
    Case (i) Suppose $M$ has a column whose entries are all equal to a positive real number, say $c$. We shall call such a column a constant column and denote it by $\vec{c}$. Let the first column of $M$ be a constant column. If not, the columns of $M$ can be interchanged to make this column the first column of $M$. Let $\Vec{s_1}, \Vec{s_2,}, \cdots, \Vec{s_u}$ denote the rows of $B$. Assume first that $v=1$. For $i \in \{1,2,\cdots, u\}$, define $f_i: \ben \rightarrow \ber $ by $f_i(n)=\Vec{s_i}.\Vec{b}(n)$. Then $f_i \in \mathcal{R}_0$. If $c \ge 1$, by Lemma \ref{l6}, for  $F \in \mathcal{P}_f(\mathcal{R}_0)$ and $\delta > 0, \exists \; m \in (0,\delta)$ and $H \in \mathcal{P}_f(\mathbb{N})$ such that for each $f_i \in F,m +  \sum_{t \in H}f_i(t) \in C$. Choose $\Vec{x}=(\frac{m}{c})$ and $\Vec{y}= \sum_{t \in H}\Vec{b}(t)$. If $0<c<1$, then for $c\delta >0$, by Lemma \ref{l6} for  $F \in \mathcal{P}_f(\mathcal{R}_0), \exists \; m' \in (0,c\delta)$ and $H' \in \mathcal{P}_f(\mathbb{N})$ such that for each $f_i \in F,m' +  \sum_{t \in H'}f_i(t) \in C$. Choose $\Vec{x}=(\frac{m'}{c})$ and $\Vec{y}= \sum_{t \in H'}\Vec{b}(t)$. The conclusion then holds.
    
    Now, assume that $v>1$. Let  $   M = \begin{pmatrix}
      c & a_{12} \cdots & a_{1v} \\
      \vdots & \vdots & \vdots \\
      c & a_{u2} \cdots & a_{uv}
      \end{pmatrix}$. \\
     Let $\delta>0$. Since, $\{\vec{b}(n)\}_{n=1}^\infty$  converges to zero, therefore  there exists a subsequence $\{\vec{b}(n_k)\}_{k=1}^{\infty}$ of $\{\vec{b}(n)\}_{n=1}^\infty$ in $(0,\frac{\delta}{2^k})^d $ such that for any $H \in \mathcal{P}_f(\mathbb{N}),$ $\sum_{k \in H} \vec{b}(n_k) \in (0, \delta)^d$. Let $\vec{b}(n)=(b_1(n), b_2(n),...,b_d(n)).$\\
      Define $f_i: \ben \rightarrow \ber$ as $f_i(t)=b_1(n_t)\sum_{j=2}^v a_{ij} + \Vec{s_i}\cdot \Vec{b}(n_t)$ for $i=\{1,2,...,u\}$. Then $\{f_i(n)\}_{n=1}^{\infty}$ converges to zero. \\
     By Lemma \ref{l6}, for $\delta > 0, \exists \; m \in (0,\delta)$ and H $\in \mathcal{P}_f(\mathbb{N})$ such that for each $i \in \{1,2,\cdots, u\}, m +  \sum_{t \in H}f_i(t) \in C$.\\
     The case for $c \ge 1$ is given below. The case for $0<c<1$ can be done in a similar manner as discussed above for the case $v=1$.\\
        Let
     $\Vec{x}=\begin{pmatrix}
         \frac{m}{c}\\
          \sum_{t\in H}b_1(n_t)\\
         \vdots\\
        \sum_{t\in H}b_1(n_t)\\
     \end{pmatrix}$ and $\Vec{y}= \sum_{t \in H}\vec{b}(n_t)$.\\
     Then, 
   \begin{align*}   M\Vec{x} + B\Vec{y} &=  \begin{pmatrix}
      c & a_{12} \cdots & a_{1v} \\
      \vdots & \vdots & \vdots \\
      c & a_{u2} \cdots & a_{uv}
      \end{pmatrix}
      \begin{pmatrix}
       \frac{m}{c}\\
          \sum_{t\in H}b_1(n_t)\\
         \vdots\\
        \sum_{t\in H}b_1(n_t)\\
      \end{pmatrix} + 
      \begin{pmatrix}
          \Vec{s_1}\\
          \vdots\\
          \Vec{s_u}
      \end{pmatrix}
      \sum_{t \in H}\Vec{b}(n_t)\\
    &= \begin{pmatrix}
      m + (\sum_{t \in H}b_1(n_t))\sum_{j=2}^{v}a_{1j} + \Vec{s}_1\cdot \sum_{t \in H}\vec{b}(n_t)\\
      \vdots \\
      m + (\sum_{t \in H}b_1(n_t))\sum_{j=2}^{v}a_{uj} + \Vec{s}_u\cdot \sum_{t \in H}\vec{b}(n_t)
      \end{pmatrix}
      \hspace{1 cm}\\
      &=
       \begin{pmatrix}
      m+ \sum_{t \in H} f_1(t) \\
      \vdots \\
       m+ \sum_{t \in H} f_u(t)
      \end{pmatrix} 
      \end{align*} 
      Thus, all entries of $M\Vec{x} + B\Vec{y}$ are in $C$\\
Case (ii) Suppose $M$ does not have a constant column, then, if $u=1$, the result is true from case(i). So assume $u>1$.  If $v=1$, then $M$ being a first entries matrix has no row equal to $\Vec{0}$ so $M$ consists of a single column, a constant column which is case (i) so we assume $v>1$. If the first column of $M$ is a zero column, we can interchange it with a non zero column of $M$, thus we may further assume the first column of $M$ is not identically zero either. \\
We now assume that the result holds for all matrices whose number of rows is less than $u$. We can write $M$ in block form as  $M = \begin{pmatrix}
      \Vec{0} & D \\
        \Vec{c} & E
      \end{pmatrix}$ 
      where $\Vec{0}$ is $w \times 1$ zero matrix, $\Vec{c}$ is $(u-w) \times 1$ constant matrix, $D$ is $w \times (v-1)$ first entries matrix with entries from $\ber$ for a positive integer $w<u$, $E$ is $(u-w) \times (v-1)$ matrix with entries from $\ber$. \\
      Let $B= \begin{pmatrix}
          B_1\\
          B_2
      \end{pmatrix}$ , where $B_1$ is $w \times d$ matrix, $B_2$ is $(u-w) \times d$ matrix. Since, $\{\vec{b}(n)\}_{n=1}^\infty$  converges to zero, therefore  there exists a subsequence $\{\vec{b}(k_t)\}_{t=1}^{\infty}$ of $\{\vec{b}(k)\}_{k=1}^\infty$ such that $\sum_{t=1}^\infty \vec{b}(k_t)$ converges. By induction hypotheses, and Lemma \ref{l9}, we can select a sequence   $\{H_n\}_{n=1}^{\infty}$ in  $\mathcal{P}_f(\ben)$ with max $H_n<$ min $H_{n+1}$ for every $n \in \ben$ and a sequence  $\{\Vec{z}(n)\}_{n=1}^{\infty}$ in $(0, \delta)^{v-1}$, which converges to zero such that  $FS(\{D\Vec{z}(n) + B_1. \sum_{t \in H_n}\Vec{b}(k_t)\}_{n=1}^{\infty}) \subseteq C^w$. \\
      Since $\{\Vec{z}(n)\}_{n=1}^{\infty}$ converges to zero, therefore, there exists a subsequence $\{\Vec{z}(n_k)\}_{k=1}^{\infty}$ of $\{\Vec{z}(n)\}_{n=1}^{\infty}$ such that $ \sum _{k=1}^{\infty}\Vec{z}(n_k)$ converges. In particular, for $\delta >0$ we can choose a subsequence  $\{\Vec{z}(n_k)\}_{k=1}^{\infty}$ such that  $\Vec{z}(n_k) \in (0,\frac{\delta}{2^k})^{v-1}$ for all $k>0$. Then for any $H \in \mathcal{P}_f(\ben)$, we have  $ \sum _{k \in H} \Vec{z}(n_k) \in (0,\delta)^{v-1}$.\\
      Let $\Vec{e}_{w+1}, \cdots, \Vec{e}_u$ denote the rows of $E$ and let $\Vec{s}_{w+1}, \cdots, \Vec{s}_u$ denote the rows of $B_2$.
      For each  $i \in \{w+1, \cdots, u\}$, we define $g_i: \ben \longrightarrow \ber$ by $g_i(n)=\Vec{e_i}.\Vec{z}(t_n) +\Vec{s_i}. \sum_{t \in H_n}\Vec{b}(k_t)$. Since $ \sum_{t=1}^{\infty}\Vec{b}(k_t)$ converges, therefore $\{ \sum_{t \in H_n} \Vec{b}(k_t)\}_{n=1}^{\infty}$ converges to $\Vec{0}$. 
      By Lemma \ref{l6}, for $\delta > 0, \exists \; m \in (0,\delta)$ and  $K \in \mathcal{P}_f(\mathbb{N})$ such that for each $i \in \{w+1,w+2, \cdots, u\}$, $m +  \sum_{n \in K}g_i(n) \in C$. The case for $c \ge 1$ is given below. The case for $0<c<1$ can be done in a similar manner as was discussed in case (i).\\
Choose $\Vec{x}=\begin{pmatrix}
          \frac{m}{c}\\
         \sum_{n \in K}\Vec{z}(t_n)
      \end{pmatrix}$ and $\Vec{y}= {\sum_{n \in K} \sum_{t \in H_n }\vec{b}(k_t)}$. Clearly $\vec{x} \in (0,\delta)^v$.  Then, \begin{align*}
          M\Vec{x}+B\Vec{y} &= 
           \begin{pmatrix}
              \Vec{0} & D \\
              c & \Vec{e}_{w+1}\\
              \vdots & \vdots \\
              c & \Vec{e}_{u}
          \end{pmatrix}
          \begin{pmatrix}
              \frac{m}{c} \\
              \sum_{n \in K}\Vec{z}(t_n)\\
          \end{pmatrix}
          +
          \begin{pmatrix}
              B_1 \\
              \Vec{s}_{w+1}\\
              \vdots \\
              \Vec{s}_u
          \end{pmatrix}
          \begin{pmatrix}
              \sum_{n \in K}\sum_{t \in H_n}\Vec{b}(k_t)
          \end{pmatrix} \\
          &= \begin{pmatrix}
              \sum_{n \in K}(D\Vec{z}(t_n) + B_1\sum_{t \in H_n} \Vec{b}(k_t) )\\
              m + \Vec{e}_{w+1} \sum_{n \in K}\Vec{z}(t_n) + \Vec{s}_{w+1}\cdot \sum_{n \in K}\sum_{t \in H_n}\Vec{b}(k_t)\\
              \vdots \\
              m + \Vec{e}_{u} \sum_{n \in K}\Vec{z}(t_n) + \Vec{s}_{u}\cdot \sum_{n \in K}\sum_{t \in H_n}\Vec{b}(k_t)
          \end{pmatrix} \\
          &= \begin{pmatrix}
              \sum_{n \in K}\Big(D \cdot \Vec{z}(t_n) + B_1 \sum_{t \in H_n}\Vec{b}(k_t)\Big) \\
              m + \sum_{n\in K}g_{w +1}(n)\\
              \vdots \\
              m + \sum_{n \in K}g_u(n)
          \end{pmatrix}
      \end{align*} 
      Therefore, all entries are in $C$. \\
      
      Case(iii) Finally, assume that $M$ is any $u \times v$ matrix with entries in $\ber$ which is $IPR/\ber^+$. By \cite[Theorem 4.1]{H} , there exists  $m \in \{1,2, \cdots,u \}$ and a  $v \times m$ matrix $G$ with non negative entries and no row equal to $\Vec{0}$ such that $MG$ is a first entries matrix with all of its first entries equal to 1.
      
      Let $G= \begin{pmatrix}
             g_{11} \cdots & g_{1m} \\
      \vdots & \vdots  \\
       g_{v1} \cdots & g_{vm}  
      \end{pmatrix}$ and let $g=$max$\{g_{ij}: i=1,2,\hdots, v, j=1,2,\hdots, m\}$. Let $\delta>0$. Since $MG$ is a first entries matrix, for $\frac{\delta}{gm}>0$, there exists $\vec{x} \in (0,\frac{\delta}{gm})^m, \vec{y} \in FS(\{\vec{b}(n)\})$ such that all entries of $MG\vec{x}+B\vec{y}$ are in $C$. \\
      Let $\vec{x}= \begin{pmatrix}
          x_1\\
          \vdots\\
          x_m
      \end{pmatrix}$ and $G\vec{x}=\vec{z}$. Then $\vec{z}=G\vec{x}=\begin{pmatrix}
           \sum_{i=1}^m g_{1i}x_i\\
          \vdots\\
         \sum_{i=1}^m g_{vi}x_i
      \end{pmatrix}$ and 
      
      $ \sum_{i=1}^m g_{ki}x_i < \delta$ for each $k=1,2,\cdots,v$.  Thus the result is true for any matrix $M$ which is $IPR/\ber^+$.
\end{proof}

\begin{theorem} \label{th5}
    Let $u,v \in \, \ben, M$ is $u \times v$  matrix with entries from $\ber$ which is $IPR/\ber^+$. Then for any C-set near-zero in $\ber^+, \{\Vec{x} \in (\ber^+)^{v}: M \vec{x} \in C^u\} $ is a C-set near zero in $(\ber^+)^{v}$.
\end{theorem}

\begin{proof}
    Define $\phi : (\ber ^+)^v \longrightarrow \ber^u$ by $\phi(\Vec{x}) = M\Vec{x}$. Let $C \subseteq \ber^+$ be a $C$-set near zero. Let $p$ be an idempotent in $J_0(\ber^+)$. To show that for every $C \in p, \{ x \in (\ber ^+)^v: M \Vec{x} \in C^u\}= \phi ^{-1}[C^u]$ is a $C$- set near zero in $(\ber^+)^v$, by Lemma \ref{l7}, it is enough to prove that for every C-set near zero in $\ber^+, \phi ^{-1}[C^u]$  is a $J-$set near zero in $(\ber^+)^v$.\\
    i.e., to prove that for every $F \in \mathcal{P}_f((\mathcal{R}_0^+)^v), \delta > 0 \; \exists \;  \vec{a} \in \, (0,\delta)^v$ and H $\in \mathcal{P}_f(\mathbb{N})$ such that for each $\vec{f} \in F, \vec{a} +  \sum_{n \in H}\vec{f}(n) \in \phi^{-1}[C^u]$
       i.e., for each  $\vec{f} \in F$, $M (\Vec{a} + \sum_{n \in H}\vec{f}(n)) \in C^u$\\
Let $F= \{ \vec{f_1}, \vec{f_2},..., \vec{f_m} \}$, for some $m \in \ben$ where each $\vec{f}_i$ is a function from $\ben$ to $(\ber^+)^v$ which converges to zero for each $i=\{1,2,\cdots, m\}$. Define $mu \times mv $ matrices $N$ and $B$ by \\
$N = 
    \begin{pmatrix}
        M & \cdots & M \\
        \vdots \\
        M & \cdots & M
    \end{pmatrix} \text{and} \hspace{0.8mm}
     B = \begin{pmatrix}
        M & O & \cdots & O \\
        O & M & \cdots & O \\
        \vdots \\
        O & O & \cdots & M
    \end{pmatrix}$\\
    where $O$ denotes the zero $u \times v $ matrix.\\
    Trivially the matrix $\begin{pmatrix}
        M\\
        \vdots\\
        M
    \end{pmatrix}$ is $IPR/\ber^+.$\\
    Therefore, By Theorem \ref{th3} $N$ is $IPR/\ber^+.$\\
    Define $\vec{g}: \ben \rightarrow (\ber^+)^{mv}$ by $\vec{g}(n) = \begin{pmatrix}
        \vec{f}_1(n)\\
        \vdots\\
        \vec{f}_m(n)
    \end{pmatrix}$.\\
    Now, $FS(\{\vec{g}(n)\}_{n=1}^\infty)$ is a weak IP set near zero in $(\ber^+)^{mv}$, therefore, by Theorem \ref{th4}, for $\frac{\delta}{m}>0, \exists \; \vec{x}=(\vec{x}_1, \dots, \vec{x}_m) \, \in \, (0,\tfrac{\delta}{m})^{mv}$ and $H \, \in \, \mathcal{P}_f(\ben)$ such that all entries of $N\vec{x} + B\vec{y}$ are in $C$ where $\vec{y}=\sum_{n \in H}\vec{g}(n)$.\\
    i.e., all entries of $\begin{pmatrix}
        M & \dots & M\\
        \vdots & \vdots & \vdots\\
        M & \cdots & M
    \end{pmatrix} \begin{pmatrix}
        \vec{x}_1\\
        \vdots \\
        \vec{x}_m
    \end{pmatrix} + \begin{pmatrix}
        M & O & \cdots & O \\
        O & M & \cdots & O \\
        \vdots & \vdots & \vdots \\
        O & O & \cdots & M
    \end{pmatrix}\sum_{n \in H}\vec{g}(n)$\\ are in $C$.
    i.e., all entries of $M(\vec{x}_1 + \cdots + \vec{x}_m) + M \sum_{n \in H}\vec{g}(n)$ are in $C$\\
    i.e., all entries of $M(\vec{z} + \sum_{k \in H}\vec{f}_i(n))$ are in $C$ for every $i \in \{1,2, \dots, m\}$, where $\vec{z}=\vec{x}_1 + \cdots + \vec{x}_m  \in (0,\delta)^v$.
    \end{proof} 

We are now in a position to prove the following theorem which was stated in the introduction in Theorem \ref{th2}.

\begin{theorem}
    Let u,v $\in \ben$ and let M be a $u \times v$ matrix with entries from $\ber$. The following statements are equivalent.\\
    (a) M is image partition regular near zero over $\mathbb{R}^+$.\\
    (b) For every C-set near zero C in $\ber^+$, there exists $\Vec{x} \in (\ber^+)^v$ such that $M\Vec{x} \in C^u$. \\
    (c) For every C-set near zero C in $\ber^+$, $\{\Vec{x} \in (\ber^+)^v : M\Vec{x} \in C^u\}$ is a C-set near zero in $(\ber^+)^v$.
\end{theorem}
\begin{proof}
   (b) $\implies$ (a). Let $r \in \ben, \bigcup_{i=1}^rC_i$ be a finite partition of $\ber^+$ and $\epsilon >0$. Now, $C$-sets near zero are partition regular, some $C_i$ is a $C$ set near zero and so $C_i \cap (0,\epsilon)$ is also a $C$  set near zero. So , there exists some $\Vec{x} \in (\ber^+)^v$ such that $M\Vec{x} \in (C_i \cap (0, \epsilon))^u$. Thus, $M$ is image partition regular near zero over $\ber^+$. (c) implies (b) is trivial . The fact that (a) implies (c) is Theorem \ref{th5}.
\end{proof} 

Theorem \ref{th4} gives a characterisation of $C$-sets near zero in $\ber^+$ which is the following.

\begin{theorem}{\label{thD}}
    Let $C \subseteq \ber^+ $ and $p$ be an idempotent in $0^+( \ber^+)$ such that $C \in p$. Let $u,v,d \in \ben$. Let $D \in p$ and let $M$ be a $u \times v$ matrix with entries from $\mathbb{R}$ which is $IPR/\ber^+$, $B$ be a $u \times d$ matrix with entries from $\mathbb{R}$ and $U$ be a weak $IP$-set near zero in $(\ber^+)^d$. Then $D$ satisfies the condition that for any $ \delta >0$, then there exists $\Vec{x} \in  (0, \delta)^v$ and $\Vec{y} \in U$ such that all the entries of $M\Vec{x} + B\Vec{y}$ are in $D$ if and only if $C$ is a $C$ set near zero.   
\end{theorem}

\begin{proof}
   The sufficiency is Theorem \ref{th4}. For the necessity, it is enough to prove that for all $D \in p, D$ is a J-set near zero. Let $u \in \ben$ and let $F=\{f_1,f_2,...,f_u\}$ be functions from $\ben$ to $\ber^+$ which converges to zero.
    Define $\vec{g}: \ben \rightarrow (\ber^+)^u$ by $\vec{g}(n) = \begin{pmatrix}
        f_1(n) \\ f_2(n) \\ \vdots \\ f_u(n) 
    \end{pmatrix}$
    for every $n \in \ben$. Let $M$ denote the $u \times 1$ matrix whose entries are all 1, and let $B$ denote the identity $u \times u$ matrix. Since $FS(\{\vec{g}(k)\}_{k=1}^{\infty})$ is a weak $IP$ set near zero in $(\ber^+)^u$, it follows that there exists $x \in (0,\delta)$ and $H \in \mathcal P_f(\ben)$ such that $Mx+B \sum_{k \in H}\vec{g}(k) \in D^u$. i.e. $x + \sum_{n \in H}f_i(n) \in D$ for every $i \in \{1,2, \dots, u\}$. So $D$ is a J-set near zero in $\ber^+$.
\end{proof}

\bibliographystyle{amsplain}

\end{document}